\documentclass[12pt]{article}
\usepackage{amsmath}
\usepackage{color}
\usepackage{amsfonts}
\usepackage{amssymb}
\usepackage{amsthm}
\usepackage{verbatim}
\usepackage{scrextend}
\usepackage{changepage}
\usepackage{lineno}

\newtheorem{theorem}{Theorem}
\newtheorem{definition}[theorem]{Definition}
\newtheorem{lemma}[theorem]{Lemma}

\newtheorem{conjecture}[theorem]{Conjecture}
\newtheorem{corollary}[theorem]{Corollary}
\newtheorem{proposition}[theorem]{Proposition}
\newtheorem{question}[theorem]{Question}
\usepackage[margin=1in]{geometry}
\newcommand{\by}{\hspace{0.02in}{\scriptstyle{\square}}\hspace{0.02in}}
\usepackage{tikz}
\usetikzlibrary{arrows}

\title{Disjoint Dominating Sets with a Perfect Matching}
\author{William F. Klostermeyer\\School of Computing\\University of North Florida\\Jacksonville, FL 32224-2669\\{\small wkloster@unf.edu}
\and Margaret-Ellen Messinger\\Department of Mathematics and Computer Science\\Mount Allison University, Sackville, NB, Canada\\{\small mmessinger@mta.ca}
\and Alejandro Angeli Ayello\\Department of Combinatorics and Optimization\\University of Waterloo, Waterloo, ON, Canada\\{\small aloangeli@gmail.com }}

\begin{document}
\maketitle

\begin{abstract}
In this paper, we consider dominating sets $D$ and $D'$ such that $D$ and $D'$ are disjoint and there exists a perfect matching between them. Let $DD_{\textrm{m}}(G)$ denote the cardinality of smallest such sets $D, D'$ in $G$ (provided they exist, otherwise  $DD_{\textrm{m}}(G) = \infty$). This concept was introduced in [Klostermeyer et al., {\it Theory and Application of Graphs}, 2017] in the context of studying a certain graph protection problem.   We characterize the trees $T$ for which $DD_{\textrm{m}}(T)$ equals a certain graph protection parameter and for which $DD_{\textrm{m}}(T) = \alpha(T)$, where $\alpha(G)$ is the independence number of $G$.  We also further study this parameter in graph products, e.g., by giving bounds for grid graphs, and in graphs of small independence number.
\end{abstract}

\section{Introduction}\label{sec:intro}

Let $G=(V, E)$ be an undirected graph. A \emph{dominating set }of graph $G$ is a set $D\subseteq V$ such that for each $u\in V-D$, there exists an $x \in D$ adjacent to $u$.
The minimum cardinality amongst all dominating sets of $G$ is the \emph{domination number}, $\gamma(G)$. Two dominating sets $D, D'$ are called {\it disjoint dominating sets} if $D \cap D' = \emptyset$. Disjoint dominating sets have been studied in the literature, see for example \cite{hed, henning}, and in particular, the related concept of disjoint maximal independent sets was considered in \cite{erdos}.

Let $DD_{\textrm{m}}(G)$ denote the cardinality of the smallest disjoint dominating sets $D$ and $D'$ of $G$ such that there is a perfect matching between them. If $G$ has no such sets, take $K_{1, m}$ with $m \geq 2$ for example, then define $DD_{\textrm{m}}(G) = \infty$. As a simple example, it is easy to see that $DD_{\textrm{m}}(P_4)=2$. The parameter $DD_{\textrm{m}}(G)$ was introduced in \cite{EvictionPaper} in the context of a certain graph protection problem that we define below.  It is important to note that the perfect matching exists {\it between} the sets $D$ and $D'$ in graph $G$ whereas $G$ itself, may or may not have a perfect matching.  

Various {\it graph protection} models have been considered in the literature, many of which are surveyed in \cite{survey}. In a graph protection problem, mobile guards aim to defend a graph from a sequence of attacks.  The first of these models was motivated by the distribution and movement of field armies during the decline of the Roman Empire~\cite{Roman}. In most graph protection models, guards occupy vertices of a dominating set $D$ on a graph $G$.  When a vertex is attacked, the guards defend against the attack by  ``moving'' to occupy a dominating set $D'$.  In a ``move'', each guard moves to a vertex in the closed neighborhood of its currently occupied vertex.  In the {\it $m$-eternal domination model}, a vertex $v \in V(G)$ is attacked and it is required that $v \in D'$ (i.e., a guard must move to $v$) whereas in the {\it $m$-eternal eviction model}, a vertex $v \in D$ is attacked and it is required that $v \notin D'$ (i.e., a guard must move off of $v$). In both of these models, the sequence of attacks is infinitely long. The $m$-eternal domination model was introduced in \cite{GHH} and the $m$-eternal eviction model was introduced in \cite{EvictionOnTrees}. Let $e_{\textrm{m}}^\infty(G)$ be the {\it eviction number} of $G$, which denotes the minimum number of guards required to defend against any sequence of attacks in the $m$-eternal eviction model. For example, $e_{\textrm{m}}^\infty(P_3)=2$.

%These dynamic domination problems can be thought of as a two-player game: the {\it defender} initially chooses a dominating set and then the players alternate turns beginning with the attacker.  The defender wins if she can defend against any sequence of attacks.  For more on eternal domination and eviction, see the survey~\cite{survey}.

%The {\it $m$-eternal domination number} $\gamma_m^\infty(G)$ and the {\it $m$-eternal eviction number} $e_{\textrm{m}}^\infty(G)$ are the minimum number of guards required to defend against any sequence of attacks in the respective models.  These parameters have been considered with respect to grid graphs (see~\cite{FMvB, GKM, EvictionPaper, MD}).

The $m$-eternal domination and $m$-eternal eviction models have been studied on grid graphs in~\cite{FMvB, GKM, EvictionPaper, MD}.  For the $m$-eternal eviction model, it was observed~\cite{EvictionPaper} that for most $2 \times n$, $3 \times n$, and $4 \times n$ grid graphs, only two disjoint dominating sets are needed to optimally defend against any sequence of attacks.  In other words, there exist dominating sets $D$ and $D'$ such that if guards occupy the vertices of $D$, then no matter which vertex in $D$ is attacked, the guards can move to occupy the vertices of $D'$, where $D \cap D' = \emptyset$.  Thus, we are motivated to ask for which graphs $G$ do such dominating sets $D$ and $D'$ exist? And when such sets exist, what is the minimum cardinality of such sets?

For a graph $G$, we define a dominating set $D$ to be a {\it swap set} if there exists a dominating set $D'$ such that there is a perfect matching between $D$ and $D'$. Again, we note that $G$ may or may not have a perfect matching and that we are simply concerned with $D$ and $D'$ having a perfect matching between them in $G$.  Note that implicitly, $D \cap D' = \emptyset$.  The {\it swap number} of graph $G$ is the minimum cardinality of a swap set on graph $G$, provided one exists. In other words, the swap number is equal to $DD_{\textrm{m}}(G)$.  Since any swap set is an $m$-eternal eviction set (as observed in \cite{EvictionPaper}), for any graph $G$, $DD_{\textrm{m}}(G)\geq e_{\textrm{m}}^\infty(G)$.

\begin{figure}[htbp]
\[ \includegraphics[width=0.3\textwidth]{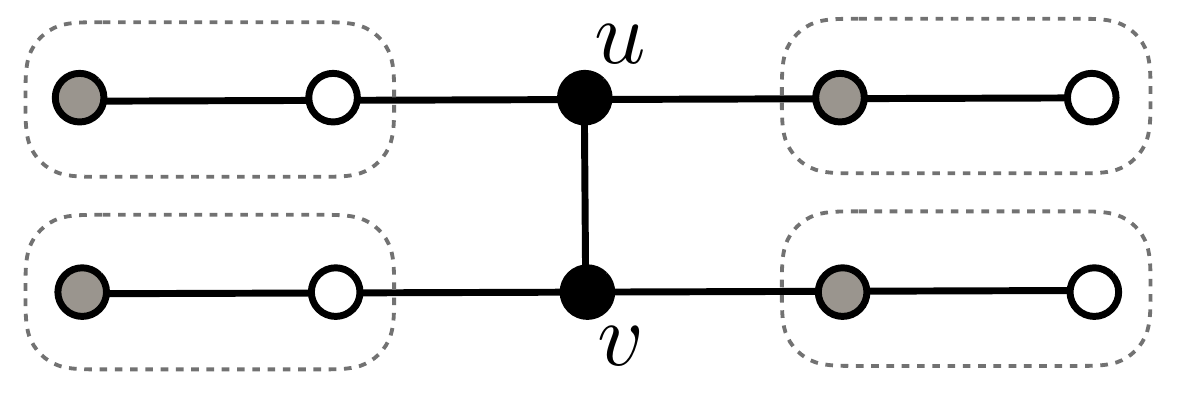} \]

\label{fig}

\caption{Disjoint dominating sets with a perfect matching indicated by the dotted lines.}
\end{figure}

Let $\alpha(G)$ denote the independence number of $G$. Let $D$ denote the set of vertices colored white and $D'$ denote the set of vertices colored gray in the tree $T$ shown in Figure~\ref{fig}; then $D$ and $D'$ are swap sets and $DD_{\textrm{m}}(T) = 4$; additionally, observe $\gamma(T) = e_{\textrm{m}}^\infty(T) = DD_{\textrm{m}}(T)$.  Observe that $D \cup \{u\}$ and $D' \cup \{v\}$ are also swap sets, but have cardinality $5$.  As a second example, consider $K_{1,3}$.  Certainly, no swap set exists in this graph: $e_{\textrm{m}}^\infty(K_{1,3}) = 3 < DD_{\textrm{m}}(K_{1,3})=\infty$. For an example of a graph with minimum degree two and no swap set, take $K_3$, duplicate each edge and then subdivide each edge. For this graph $G$, $\alpha(G)=6$, $|V(G)|=9$, $DD_{\textrm{m}}(G) = \infty$.

Finally, the graph $G$ shown in Figure~\ref{Fig2} illustrates there may be strict inequality with respect to the eviction and the swap numbers: $e_{\textrm{m}}^\infty(G) = 3 < 4 = DD_{\textrm{m}}(G)$.  For the $m$-eviction model, we simply place a guard at each of the gray vertices in Figure~\ref{Fig2}.  After the guards move in response to any attack, at least two of the three guards can occupy gray vertices.  Then $3$ guards are both sufficient and necessary (as $\gamma(G)=3)$ to defend any sequence of attacks and $3 = e_{\textrm{m}}^\infty(G) \leq DD_{\textrm{m}}(G)$.  Suppose $DD_{\textrm{m}}(G) = 3$ and let $D$, $D'$ be minimum swap sets with a perfect matching between them.  Observe that each of $D$, $D'$ contains exactly one vertex from each $3$-cycle.  Without loss of generality, $a_1 \in D$ (then $a_1 \notin D'$).  Then $a_2,a_3 \in D'$, otherwise the vertices on the $6$-cycle will not be dominated by set $D'$. However, not all the vertices of the $6$-cycle will be dominated by $D$.  Therefore, $DD_{\textrm{m}}(G) \geq 4$.  The reader may easily verify that $DD_{\textrm{m}}(G) = 4$.  Thus, $3 = e_{\textrm{m}}^\infty(G) < DD_{\textrm{m}}(G) = 4$.

\begin{figure}[htbp]
\[ \includegraphics[width=0.35\textwidth]{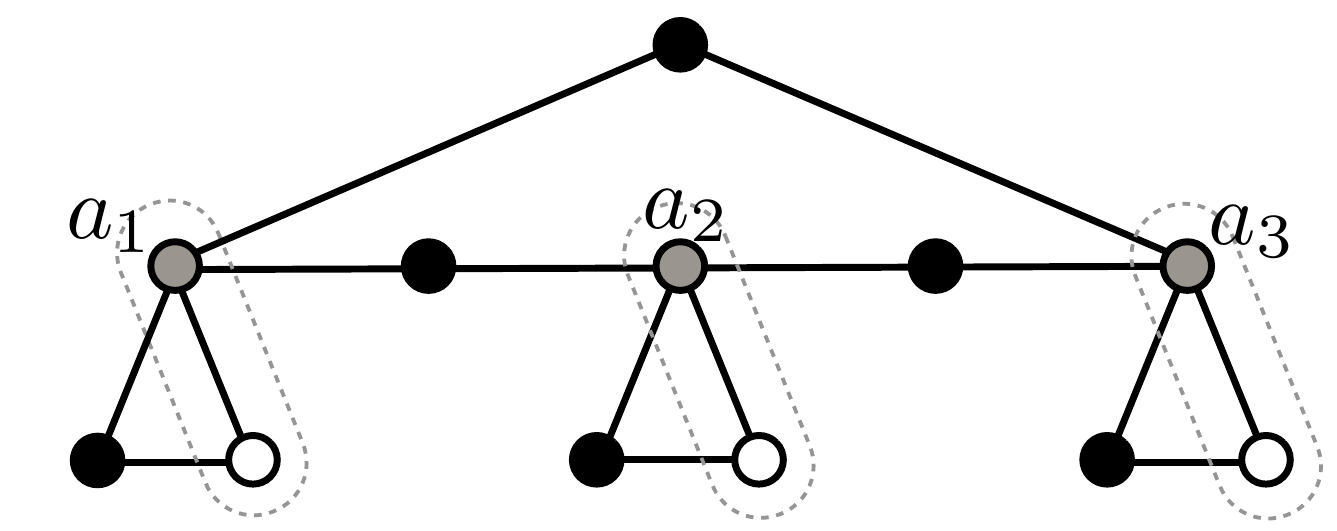} \]

\caption{A graph $G$ for which $e_\textrm{m}^\infty(G) = 3 < 4 = DD_{\textrm{m}}(G)$.}

\label{Fig2}
\end{figure}

Observations about Figures~\ref{fig} and~\ref{Fig2} lead to the following natural question.

\begin{question} Can we characterize the graphs $G$ for which $e_{\emph{m}}^\infty(G) = DD_{\emph{m}}(G)$? \end{question}

We begin with introductory results in Section~\ref{definitions}, pointing out a minor mistake in the definition of star partitionings used in~\cite{EvictionOnTrees} and~\cite{KlosterMyn} that were used to determine $e_{\textrm{m}}^\infty$ for trees.  In Section~\ref{sec:trees}, we use star partitionings to show that for certain trees, $DD_{\textrm{m}}(T)$ and $e_{\textrm{m}}^\infty(T)$ are equal.  This provides an alternate proof for the result of~\cite{EvictionOnTrees} that shows the $m$-eternal eviction number equals the star partitioning number.   In Section~\ref{sec:ind}, we characterize trees for which $DD_{\textrm{m}}(G)$ equals $\alpha(G)$. In Section~\ref{sec:products}, we consider graph products, and return to the graphs whose eviction number motivated the definition of the swap number: grid graphs.  We show the swap number of any $m \times n$ grid graph is $mn/5+o(m+n)$ and determine the swap number for some small grid graphs. We also provide an upper bound for $DD_{\textrm{m}}(G \by H)$ for any non-trivial graphs $G$ and $H$ --- regardless of whether $G$ and $H$ have swap sets.  Finally, in Section~\ref{sec:ind2}, we consider graphs with independence number at most three and pose two conjectures.

%%%%%%%%%%%%%%%%%%%%%%%%%%%%%%%%

\section{Definitions and Preliminary Results}\label{definitions}

%%%%%%%%%%%%%%%%%%%%%%%%%%%%%%%%

Denote the open and closed neighborhoods of a vertex $x\in V$ by $N(x)$ and
$N[x]$, respectively. That is, $N(x)=\{v|xv\in E\}$ and $N[x]=N(x)\cup\{x\}$.
Further, for $S\subseteq V$, let
$N(S)=\bigcup_{x\in S}N(x)$. %For any $X\subseteq V$ and $x\in X$, we say that $v\in V-X$ is an \emph{external private neighbor }of $x$ \emph{with respect to} $X$ if $v$ is adjacent to $x$ but to no other vertex in $X$. The set of all such vertices $v$ is the \emph{external private neighborhood of }$x$ \emph{with respect to}$X$.

A {\it star} is the complete bipartite graph $K_{1, n}$ for some $n \geq 1$. The {\it center} is either the unique vertex of degree greater than one or, in the case of $K_{1, 1}$ an arbitrarily selected vertex.
A graph with $n$ vertices is {\it trivial} if $n \leq 1$ and {\it non-trivial} otherwise.

We state a definition from \cite{EvictionOnTrees}.

\begin{definition}\label{defn} \cite{EvictionOnTrees} A {\it star partitioning} $P$ of a non-trivial tree $T$ is a partition of the vertices of $T$ such that \vspace{0.1in}

(i) each part of $P$ induces a star, and

(ii) no two $K_1$ parts are adjacent.\vspace{0.1in}

A star partitioning of a non-trivial tree $T$ is a special star partitioning if\vspace{0.1in}

(iii) each vertex of $T$ adjacent to exactly one leaf forms a $K_2$ part of $P$ with this leaf, and

(iv) any $K_1$ part is adjacent to at least two parts that are not $K_1$'s.\end{definition}

A star $S$ of order $k \geq 1$ in a special star partitioning is assigned {\it weight} $k-1$.  The {\it weight} $\omega(P)$ of a special star partitioning $P$ is the sum of the weights of its parts and $s(T)$ is the minimum weight of any special star partitioning of $T$.

\begin{proposition}\label{prop} [Proposition 28 in~\cite{EvictionOnTrees}] For any tree $T$ with at least two vertices, $s(T) = e_{\emph{m}}^\infty(T)$.\end{proposition}

Consider once again, the tree $T$ given in Figure~\ref{fig} and recall that $DD_{\textrm{m}}(T) = e_{\textrm{m}}^\infty(T)=4$: the lower bound follows as $\gamma(T) = 4$.  Let the gray colored vertices form set $D$ and the white colored vertices form set $D'$.  Then $DD_{\textrm{m}}(T)=4$.  For eviction, the guards simply move between sets $D$ and $D'$ in response to attacks.  However, we observe that following the definition for a special star partitioning, given by~\cite{EvictionOnTrees} and~\cite{KlosterMyn}, we find that $s(T) = 5$.  By Condition (iii) of Definition~\ref{defn}, there are (at least) four $K_2$ parts in Figure~\ref{fig}; these are indicated by dotted lines.  Following the Condition (ii), the remaining two vertices (colored black) must form a $K_2$.  Thus, $s(T) =5 \neq 4 = e_{\textrm{m}}^\infty(T)$.

We observe, however, that if Condition (ii) of Definition~\ref{defn} is removed and a special star partition is allowed to have two $K_1$ parts adjacent, then the weight of the special star partitioning for the tree $T$ in Figure~\ref{fig} is reduced to $4$ and then equals $e_{\textrm{m}}^\infty(T)$ (in this case the black vertices are each $K_1$ parts with weight $0$).  More generally, one may observe that if Condition (ii) of Definition~\ref{defn} is removed, then the proof of Proposition~\ref{prop} (Proposition 28 in~\cite{EvictionOnTrees}) holds (as it does not depend on Condition (ii) of Definition~\ref{defn}).

We give a corrected definition below, which we call a  {\it simple star partitioning} for clarity.
%It differs from Definition~\ref{defn} in that the condition that no two $K_1$ parts are adjacent has been removed. % and the condition that every part is a $K_1$ or $K_2$ part has been aswed.  In Section~\ref{alg}, we describe an algorithm to determine $DD_{\textrm{m}}(T)$ for all trees that uses star partitionings; however since we are considering disjoint dominating sets with perfect matchings between them, we restrict the stars in the partition to $K_1$'s and $K_2$'s.  In disjoint dominating sets $D$ and $D'$ with perfect matchings, each vertex in $D$ is paired with a vertex in $D'$ (forming a $K_2$ part) or is in neither $D$ nor $D'$ (forming a $K_1$ part).

\begin{definition}\label{defn1} A {\it simple star partitioning} $P$ of a non-trivial tree $T$ is a partition of the vertices of $T$ such that \vspace{0.1in}

(i) each part of $P$ induces a star,

(ii) each vertex of $T$ adjacent to exactly one leaf forms a $K_2$ part of $P$ with this leaf, and

(iii) any $K_1$ part is adjacent to at least two parts that are not $K_1$'s.\end{definition}

The weight $S(T)$ of a simple star partitioning $P$ is defined analogously to that of a special star partitioning.

Thus, with the change in definition, Proposition~\ref{prop} [Proposition~28 in~\cite{EvictionOnTrees}] becomes the following.

\begin{proposition}\label{prop1} For any tree $T$ with at least two vertices, $S(T) = e_{\emph{m}}^\infty(T)$.
\end{proposition}

\begin{definition}
A {stem} is a vertex adjacent to a vertex of degree one. A {weak stem} is a vertex adjacent to exactly one vertex of degree one and a {strong stem} is a vertex adjacent to two or more vertices of degree one.  A weak graph is a graph with no strong stems whereas a strong graph is a graph with at least one strong stem.
\end{definition}

\begin{lemma}\label{lemma:fwd_sw_set_characterisation}
If $G$ is a strong graph then $DD_{\emph{m}}(G) = \infty$.
\end{lemma}

\begin{proof}
Let $G$ be a strong graph and assume, by way of contraction, that $DD_{\textrm{m}}(G)=k$ for some integer $k$.  Let $D$ and $D'$ be minimum disjoint dominating sets with a perfect matching between them.  Since $G$ is a strong graph, $\exists~v \in V(G)$ such that $v$ is adjacent to leaves $\ell_1$ and $\ell_2$. Note that at least one of $v,\ell_1$ is in $D$; else $\ell_1$ is not dominated.  If $v \in D$, then $v \notin D'$.  Then $\ell_1, \ell_2 \in D'$, else $D'$ does not dominate all vertices of $G$.  However, at most one of $v\ell_1$, $v\ell_2$ is in the perfect matching between $D$ and $D'$ and we have a contradiction (as one of $\ell_1,\ell_2$ is not paired with any vertex of $D$).  If $v \notin D$, we have a contradiction via a similar argument.\end{proof}

Finally, we state a simple, but useful result.  Suppose $G$ has a spanning subgraph $H$ and where $H$ has a swap set.  Then let $D,D'$ be disjoint dominating sets on $H$ with a perfect matching.  Certainly $D,D'$ are also disjoint dominating sets on $G$ with a perfect matching as every edge of $H$ is also in graph $G$.

\begin{lemma}\label{subgraph}
Let $G$ be a graph.  If $G$ has a spanning subgraph $H$ and $H$ has a swap set, then $DD_{\emph{m}}(G) \leq DD_{\emph{m}}(H)$.\end{lemma}

%%%%%%%%%%%%%%%%%%%%%%%%%%%%%%%%%%%%%%%%%%%%%%%%%%

\section{Trees}\label{sec:trees}

%%%%%%%%%%%%%%%%%%%%%%%%%%%%%%%%%%%%%%%%%%%%%%%%%%

\subsection{Basic Results on Trees}\label{subsec:sw}

%%%%%%%%%%%%%%%%%%%%%%%%%%%%%%%%%%%%%%%%%%%%%%%%%%

In this section, we show that if $T$ is a non-trivial tree with $DD_{\textrm{m}}(T)=k$ for some finite $k$, then $DD_{\textrm{m}}(T) = e_{\textrm{m}}^\infty(T) = S(T)$. From Lemma~\ref{lemma:fwd_sw_set_characterisation}, we know that if $T$ is not a weak tree (i.e., if $T$ is a strong tree), then $DD_{\textrm{m}}(T)=\infty$.  Thus, we focus on weak trees.

\begin{lemma}\label{lemma:K1K2partition}
If $T$ is a non-trivial weak tree then there exists a minimum-weight simple star partitioning $P$ of $T$ such that all parts of $P$ are either $K_1$'s or $K_2$'s.
\end{lemma}

\begin{proof}
Let $T$ be a non-trivial weak tree. By way of contradiction, assume that every minimum-weight simple star partitioning of $T$ has a star part that is neither a $K_1$ nor a $K_2$ part.  Let $P$ be a minimum weight simple star partitioning of $T$ and $v$ be the central vertex of a star part $S$ in $P$ that is neither a $K_2$ nor a $K_1$.  Let $x,y$ be vertices in $S$ and observe that since $T$ is weak, at most one of $x,y$ is a leaf in $T$.  Without loss of generality, suppose $x$ is not a leaf in $T$.  If none of the vertices of $N(x)\backslash\{v\}$ are $K_1$ parts in $P$, then $P$ is not minimum: in this case,  we create a new partitioning $P'$ which differs from $P$ in that $x$ is a $K_1$ part, rather than in $S$.  Observe $P'$ remains a simple star partitioning as $x$ (a $K_1$ part) is adjacent to at least two non-$K_1$ parts.  Since the number of vertices in $S$ is one fewer in $P'$ than $P$ and a $K_1$ part contributes weight $0$, $P'$ has weight one less than $P$.  Thus, every vertex in $N(x) \backslash \{v\}$ is a $K_1$ part.  Let $u$ be a vertex in $N(x)\backslash\{v\}$.  Then we can construct a new simple star partitioning $P''$ from $P$ where $x$ and $u$ form a $K_2$ part, instead of $x$ being in $S$.  Observe that $P''$ is a simple star partitioning and the weight of $P''$ equals that of $P$.

We can repeat this same process for all star parts that are not $K_1$ or $K_2$ parts to construct a minimum-weight simple star partitioning with only $K_1$ and $K_2$ parts.
\end{proof}

\begin{theorem}\label{lemma: Evicgeqsw} If $T$ is a non-trivial weak tree then $T$ has a swap set and $DD_{\emph{m}}(T)\leq S(T)$.\end{theorem}

\begin{proof}  Let $T$ be a non-trivial weak tree and let $P$ be a minimum-weight simple star partitioning of $T$ such that all the star parts are $K_1$ parts or $K_2$ parts  (by Lemma \ref{lemma:K1K2partition} we know such partitioning exists).  We construct a swap set of cardinality $S(T)$, by associating, to the vertices in each $K_2$ part, a label of $D$ or $D'$.

Since $P$ is a simple star partitioning, every $K_1$ part is adjacent to at least two $K_2$ parts.  Let $u$ be a $K_1$ part (if one does not exist, then we skip to the last paragraph of the proof).  Considering the adjacent $K_2$ parts, if there are two neighbors $v_1,v_2$ of $u$ that have not yet been assigned to dominating sets, then we assign $v_1,v_2$ the labels of $D,D'$, respectively (the neighbors of $v_1,v_2$ in their respective $K_2$ parts are then labeled $D',D$, respectively).  If there is one neighbor $v_1$ that has not yet been assigned a label and at least one neighbor $v_2$, with an assigned label $D'$, then we assign $v_1$ a label of $D$ (and the neighbor in its $K_2$ part is labeled $D'$).  If other $K_1$ or $K_2$ parts are adjacent to $u$, then it is irrelevant to $u$, whether their vertices are labeled $D$ or $D'$.  We repeat for each $K_1$ part until all $K_1$ parts have been considered or there is a $K_1$ part that is only adjacent to vertices labeled, without loss of generality, $D$.

Suppose that $u$ is a $K_1$ part that is adjacent only to vertices labeled $D$.  Let $v_1$ be a vertex in an adjacent $K_2$ and let $T_{v_1}$ denote the tree containing $v_1$ induced by the deletion of $u$ from $T$.  We now swap the labels on all labeled vertices of $T_{v_1}$.  Now, $u$ is adjacent to a vertex labeled $D$ and a vertex labeled $D'$.

Finally, if the vertices of any $K_2$ parts remain unlabeled, one vertex is arbitrarily labeled $D$ and one is labeled $D'$.  The vertices labeled $D$ and $D'$ form disjoint dominating sets with a perfect matching.\end{proof}

The next corollaries follow directly from Proposition~\ref{prop1}, Lemma~\ref{lemma:fwd_sw_set_characterisation}, Lemma~\ref{subgraph}, and Theorem~\ref{lemma: Evicgeqsw}.

\begin{corollary}\label{theorem:characterisation_swap set}
A non-trivial tree $T$ has a swap set if and only if $T$ is a weak tree.
\end{corollary}

\begin{corollary}\label{cor}
For any non-trivial weak tree $T$, $DD_{\emph{m}}(T) = S(T) = e_{\emph{m}}^\infty(T)$.
\end{corollary}

\begin{corollary}\label{corollary: SpanningTreesw}
If $G$ has a spanning tree $T'$ such that $T'$ is a weak tree then $DD_{\emph{m}}(G) \leq S(T')$.  \end{corollary}

%Now that we have characterized the trees with finite $sw$ number, we will show $DD_{\textrm{m}}(T)=s(T)$ for all weak trees $T$.  In \cite{EvictionOnTrees}, it was shown that $e^\infty_m(T) = s(T)$ for all trees $T$ where $e^\infty_m(T)$ denotes the eviction number of the tree.

%\begin{proposition}\label{proposition:evictionTrees}\emph{\cite{EvictionOnTrees}}
%For any tree $T$ with at least two vertices, $e^\infty_m(T) = s(T)$.
%\end{proposition}

%\begin{proposition}\label{lemma: Evicleqsw}  Let $G$ be a graph then $DD_{\textrm{m}}(G)\geq e^\infty_m(G)$.
%\end{proposition}

%\begin{theorem}
%If $T$ is a weak tree of order 2 or greater then $DD_{\textrm{m}}(T)=e^\infty_m(T)=s(T)$.
%\end{theorem}

%\begin{proof}
%Follows from  Lemma \ref{lemma: Evicgeqsw},  Proposition \ref{lemma: Evicleqsw} and  Proposition \ref{proposition:evictionTrees}.
%\end{proof}

%%%%%%%%%%%%%%%%%%%%%%%%%%%%%%%%%%%%%%%%%%%%%%%%%%

%%%%%%%%%%%%%%%%%%%%%%%%%%%%%%%%%%%%%%%%%%%%%%%%%%

\subsection{Independent Sets in Trees}\label{sec:ind}

%%%%%%%%%%%%%%%%%%%%%%%%%%%%%%%%%%%%%%%%%%%%%%%%%%

In this section, we characterize the trees where $\alpha(T) = DD_{\textrm{m}}(T)$. We first exploit a result of Jou~\cite{Jou} to characterize trees for which $\gamma$, $e_{\textrm{m}}^\infty$, $DD_{\textrm{m}}$, $\alpha$ are all equal.  For a graph $G$, let $\widehat{G}$ be the graph obtained by adding a pendant (leaf) vertex adjacent to each vertex of $G$.  More formally, if $V(G) = \{v_1,v_2,\dots, v_n\}$, then $\widehat{G}$ is the graph with $V(\widehat{G}) = V(G) \cup \{u_1,u_2,\dots,u_n\}$ and $E(\widehat{G}) = E(G) \cup \{u_1v_1,u_2v_2,\dots,u_nv_n\}$.

%For a graph $G$, let $\widehat{G}$  be the graph with vertex set $V(\widehat{G}) = V (G) \cup \{\widehat{v} : v \in V (G)\}$ and the edge set $E(\widehat{G}) = E(G)\cup\{v\widehat{v} : v \in V (G)\}$.

\begin{theorem}\label{theorem: DomInCharact}\emph{\cite{Jou}} If $T$ is a non-trivial tree, then $\gamma(T) = \alpha(T)$ if and only if $T = \widehat {H}$ for some tree $H$ of order $|V(T)|/2$.
\end{theorem}

\begin{corollary}
If $T$ is a non-trivial tree then $\gamma(T) = e_{\emph{m}}^\infty(T) = DD_{\emph{m}}(T)= \alpha(T)$ if and only if $T = \widehat {H}$ for some tree $H$ of order $|V(T)|/2$.
\end{corollary}

\begin{proof} Note that all vertices in $H$ are weak stems in $T$ therefore $T$ is a weak tree and $e^\infty_m(T) = DD_{\textrm{m}}(T)$ by Corollary~\ref{cor}. From Corollary~\ref{corollary:swleqAlph} we know that $DD_{\textrm{m}}(T)\leq \alpha(T)$ and by definition $\gamma(T)\leq DD_{\textrm{m}}(T)$.
\end{proof}

We next state a result from~\cite{EvictionOnTrees} that will be useful later in this section.

\begin{theorem}\label{theorem: alphagreaterthanE}\emph{\cite{EvictionOnTrees}} Let $G$ be a connected graph. Then $e_{\emph{m}}^\infty(G) \leq \alpha(G).$ \end{theorem}

\begin{corollary}\label{corollary:swleqAlph} Let $T$ be a non-trivial weak tree. Then $DD_{\emph{m}}(T)\leq \alpha(T).$
\end{corollary}

Corollary \ref{corollary:swleqAlph}  follows immediately from Corollary~\ref{cor} and Theorem~\ref{theorem: alphagreaterthanE} \cite{EvictionOnTrees}, but it does not give us any insight into when the parameters $DD_{\textrm{m}}$ and $\alpha$ are equal. Additionally, in \cite{EvictionOnTrees}, the authors asked the following.

\begin{question}\label{q} (Question 4  in~\cite{EvictionOnTrees})

(i) Can we characterize the graphs with $e_{\emph{m}}^\infty$ equal to $\gamma$?

(ii) Can we characterize the graphs with $e_{\emph{m}}^\infty$ equal to $\alpha$? \end{question}

Question~\ref{q} is answered affirmatively for trees in~\cite{KlosterMyn}.   In particular, in answering (ii), the authors showed that for any non-trivial tree $T$, $e_{\textrm{m}}^\infty(T) = \alpha(T)$ if and only if $T$ has a minimum-weight special star partitioning containing no $K_1$ parts.  Consider again, the tree $T$ given in Figure~\ref{fig}.  Using Definition~\ref{defn}, which was stated in both~\cite{EvictionOnTrees} and~\cite{KlosterMyn}, we find $s(T) = 5$.  Following (iii) of Definition~\ref{defn}, each leaf forms a $K_2$ part with its stem (indicated by dotted lines in Figure~\ref{fig}).  Following (ii) of Definition~\ref{defn}, the vertices of degree three cannot both be $K_1$ parts.  Thus, they must form a $K_2$ part together and $s(T)=5$.  As the four leaves along with one vertex of degree $3$ form an independent set, we see $\alpha(T)=5$.  However, as shown in Section~\ref{sec:intro}, $e_{\textrm{m}}^\infty(T) = 4$, which forms a contradiction.  However, using the {\it simple} star partitioning definition given in Definition~\ref{defn1}, the proof and results of~\cite{KlosterMyn} holds.

\begin{theorem}\label{tt} For any non-trivial tree $T$, $e_{\emph{m}}^\infty(T) = \alpha(T)$ if and only if $T$ has a minimum-weight simple star partitioning containing no $K_1$ parts. \end{theorem}

We later provide an alternate proof of this theorem, while also proving that for any non-trivial tree $T$, $\alpha(T)=DD_{\textrm{m}}(T)$ if and only if $S(T) = |V(T)|/2$.

\begin{definition}
Let the \textbf{weak reduction} of a tree $T$ be the induced subgraph of $T$ obtained by removing all but one leaf from each strong stem. \end{definition}

\begin{lemma}\label{theorem: TreeswomStrogTrees}
Let $T'$ be a weak reduction of tree $T$.  Then $S(T) = S(T') + |V(T)\backslash V(T')|$.
\end{lemma}

\begin{proof} If $T \cong T'$, the result follows trivially.   Let $T$ be a tree with strong stem $v$ and $\ell > 1$ adjacent leaves, $u_1,u_2,\dots,u_\ell$.

First, let $P$ be a minimum-weight simple star partitioning of $T$.  By Definition~\ref{defn1}, $v$ must form the center of a star $S$ in $P$ and vertices $u_1,u_2,\dots,u_\ell$ must form leaves of $S$.  Note that $S$ may also contain other vertices of $N(v)$; thus $S \cong K_{1,\ell+m}$ where $m$ is the number of vertices of $N(v) \backslash \{u_1,\dots,u_\ell\}$ in $S$.  Let $T^\star$ be the tree induced by deleting $u_2,\dots,u_\ell$ from $T$; observe that by replacing $S$ with $S^\star$ where $V(S^\star) = V(S) \backslash \{u_2,\dots,u_\ell\}$, we find $S^\star \cong K_{1,m+1}$ and obtain a simple star partitioning $P^\star$ of $T^\star$.  Then $S(T^\star) \leq S(T) - w(S)+w(S^\star) = S(T)-(\ell+m)+(m+1) = S(T) - (\ell-1)$.  Repeating this argument for each strong stem of $T$ leads to $S(T') \leq S(T) - |V(T)\backslash V(T')|$.

Next, let $P'$ be a minimum-weight simple star partitioning of $T'$.  By Definition~\ref{defn1} (ii), strong stem $v$ must form a $K_2$ part with leaf $u_1$ in $T'$.  Let $T^\diamond$ be the tree with $V(T^\diamond) = V(T') \cup \{u_2,u_3, \dots,u_\ell\}$ and $E(T^\diamond) = E(T') \cup \{vu_2,vu_3,\dots,vu_\ell\}$.  We impose the partitioning $P'$ on $T^\diamond$ with the following modification: instead of $v$ forming a $K_2$ part with $u_1$,  $v$ will form the center of a $K_{1,\ell}$ star part that includes $u_1,u_2,\dots, u_\ell$.  Observe that this forms a simple star partitioning on $T^\diamond$.  Therefore $S(T^\diamond) \leq S(T') + (\ell-1)$.  Repeating this argument for each strong stem of $T$ leads to $S(T) \leq S(T') + |V(T)\backslash V(T')|$.
\end{proof}

\begin{theorem} \label{lemma:fwd-AlphaswCharc-v2} Let $T'$ be a weak reduction of non-trivial tree $T$.  Then $S(T') = |V(T')|/2$ if and only if $\alpha(T) = S(T)$.\end{theorem}

\begin{proof} First suppose $S(T')=|V(T')|/2$. Then Lemma~\ref{lemma:K1K2partition} implies there is a minimum-weight simple star partitioning $P'$ of $T'$ such that every vertex is in a $K_2$ part. From the proof of Lemma~\ref{theorem: TreeswomStrogTrees}, we can extend $P'$ in $T'$ to a simple star partitioning $P$ of $T$ such all vertices are in a $K_2$ part, except for the strong stems and their adjacent leaves which form a $K_{1,n}$ part with $n\geq 2$.  If $\alpha(T)\geq S(T)+1$, there must be a star part $S$ in $P$ such that $S$ has at least $|V(S)|$ independent vertices, which is a contradiction since no two independent vertices can be adjacent.  Therefore $\alpha(T) \leq S(T)$.  Proposition~\ref{prop1} and Theorem~\ref{theorem: alphagreaterthanE} imply $\alpha(T)\geq S(T)$, so if $S(T')=|V(T')|/2$, then $\alpha(T)=S(T)$.

We next suppose $S(T') \not =|V(T')|/2$ and prove that $\alpha(T) \not = S(T)$.  Since $T'$ is a weak tree, $S(T') \leq |V(T')|/2$ by  Lemma~\ref{lemma:K1K2partition}.  Assume $S(T') < |V(T')|/2$ and by Lemma~\ref{theorem: TreeswomStrogTrees}, $S(T)<|V(T')|/2+|V(T)\backslash V(T')|$.  Since $T$ is a tree, it is necessarily bipartite: let $X_T$, $Y_T$ be partite sets of $T$.   As $T'$ is a weak reduction of $T$, it is an induced subgraph of $T$.  Therefore, for some integer $k$, $T'$ contains $k$ vertices from $X_T$ and $|V(T')|-k$ vertices from $Y_T$.  Since $k + |V(T')|-k = |V(T')|$, at least one of $X_T,Y_T$ must contain $|V(T')|/2$ vertices.  Thus, $T'$ contains an independent set $I$ of cardinality at least $|V(T')|/2$.

If $T \cong T'$, then a contradiction has been obtained as $\alpha(T) \geq |V(T)|/2 > S(T)$.  Otherwise, consider a strong stem $u$ in $T$.  If $u \notin I$, we add all leaf neighbors of $u$ to $I$ (in this case, one leaf neighbor of $u$ existed in $T'$ and was already in $I$) and $I$ remains an independent set in $T$.   If $u \in I$, then we remove $u$ from $I$ and add all leaves of $u$ (in $T$) to $I$ and $I$ remains an independent set in $T$.  As a result, $\alpha(G) \geq |V(T')/2|+|V(T)\backslash V(T')| > S(T)$ and have achieved the desired contradiction.
\end{proof}

The next corollary follows immediately from Proposition~\ref{prop1} and Theorem~\ref{lemma:fwd-AlphaswCharc-v2}.

\begin{corollary}\label{corr}
Let $T'$ be a weak reduction of non-trivial tree $T$. Then $\alpha(T) = e_{\emph{m}}^\infty(T)$ if and only if  $S(T') = |V(T')|/2$.  \end{corollary}

%Furthermore we can see that there is a perfect matching in $T$ where the $K_2$ parts represent the matchings and that we can find an independent set $I$ such that there is one independent vertex from each $K_2$ part.  Observe that the characterization in Corollary~\ref{corr} is equivalent to that given in Theorem~\ref{tt}.

%\begin{proposition} \label{proposition: MatchingIndepS}Let $T$ be a tree and $M$ be a matching then there are two disjoint independent sets $I$ and $I'$ in $T$ such that every edge in $M$ matches one vertex from $I$ to a vertex in $I'$.\end{proposition}

\begin{theorem}\label{theorem:bwd-AlphaswCharc} For any non-trivial tree $T$, $\alpha(T) = DD_{\emph{m}}(T)$ if and only if  $T$ is a weak tree and $S(T) = |V(T)|/2$.
\end{theorem}

\begin{proof}  Suppose $T$ is a weak tree and $S(T)=|V(T)|/2$.  By Corollary~\ref{cor} and Theorem~\ref{lemma:fwd-AlphaswCharc-v2}, $DD_{\textrm{m}}(T)=S(T)=\alpha(T)$.

For the other direction, we consider the contrapositive and prove:  if $T$ is a strong tree or $S(T) \neq |V(T)|/2$ then $\alpha(T) \neq DD_{\textrm{m}}(T)$.  If $T$ is a strong tree then clearly $\alpha(T)\not=DD_{\textrm{m}}(T)$ since $DD_{\textrm{m}}(T) = \infty$ and $\alpha(T)$ is always finite.  Therefore, we assume $T$ is not a strong tree (i.e., $T$ is a weak tree) and $S(T) \neq |V(T)|/2$.  By Theorem~\ref{lemma:fwd-AlphaswCharc-v2}, $\alpha(T) \neq S(T)$.  Since $T$ is a weak tree, by Corollary~\ref{cor}, $DD_{\textrm{m}}(T) = S(T)$.  Thus, $DD_{\textrm{m}}(T) \neq \alpha(T)$.\end{proof}

\section{The Cartesian Product}\label{sec:products}

%%%%%%%%%%%%%%%%%%%%%%%%%%%%%%%%%%%

%%%%%%%%%%%%%%%%%%%%%%%%%%%%%%%%%%%
\subsection{General Results}\label{subsec:prod}

Graph products are powerful tools in graph theory, and are concerned with
taking two (or more) graphs and generating new ones.
%Products often
%lead to more complex graphs, and yet they are good instruments for probing
%the structure of graphs.
Cartesian products are one of the most studied graph products. The \emph{Cartesian} (or \emph{box}) \emph{product} of $G$ and $H,$ written $G \by H$, has vertex set $V(G)\times V(H).$ Vertices $(a,b)$ and $(c,d)$ are adjacent if $a=c
$ and $bd\in E(H),$ or $ac\in E(G)$ and $b=d.$ See \cite{ik} for further background on the Cartesian and other graph products. Several unresolved and central conjectures focus on graph products, such as those of Vizing on the domination number of Cartesian products \cite{viz,viz2}.

\begin{figure}[h]
\[\includegraphics[width=0.6\textwidth]{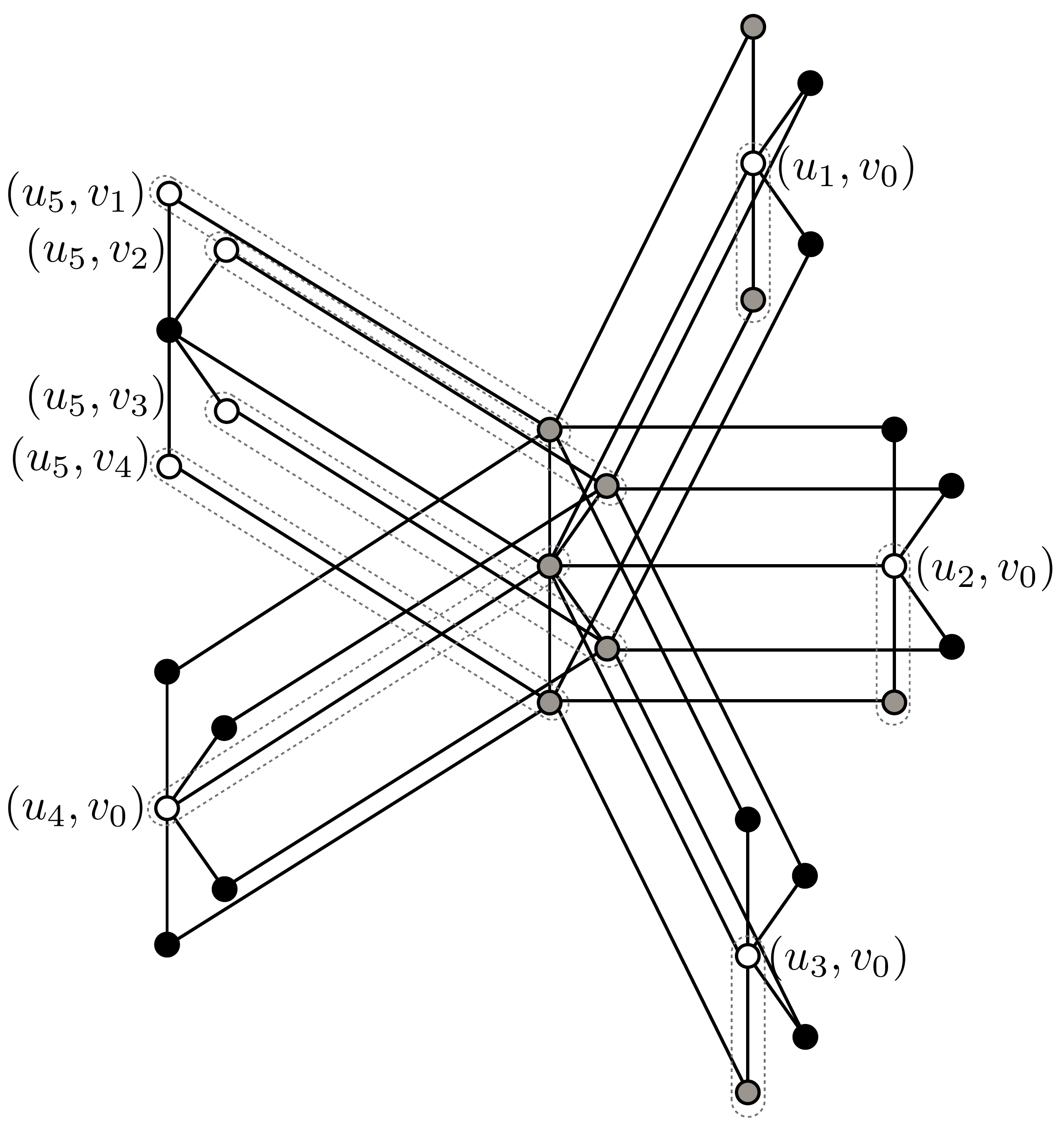}\]

\caption{A swap set for $K_{1,5} \by K_{1,4}$.}

\label{fig:prod}
\end{figure}

Consider, for example, the Cartesian product of stars $K_{1,5}$ and $K_{1,4}$, as shown in Figure~\ref{fig:prod}.  Certainly no swap set exists for either $K_{1,5}$ or $K_{1,4}$, but in the product, disjoint dominating sets with a perfect matching exist and are illustrated with white and gray vertices in Figure~\ref{fig:prod}.  Let $D$ and $D'$ denote the set of white and gray vertices, respectively.  This configuration of sets $D$ and $D'$ is exploited to find $DD_{\textrm{m}}$ for products of stars in Theorem~\ref{thm:star}, which will later be used to prove that for any graphs $G$ and $H$ having at least two vertices each, a swap set exists for $G \by H$ regardless of whether swap sets exist for $G$ and $H$.

\begin{theorem}\label{thm:star}  Let $p,q \geq 1$.  Then $DD_{\emph{m}}(K_{1,p} \by K_{1,q}) = p+q-1$. \end{theorem}

\begin{proof}  If $p=q=1$ then $K_{1,p} \by K_{1,q} \cong P_2 \by P_2 \cong C_4$ and obviously, $DD_{\textrm{m}}(C_4) =2$. Therefore, assume $p \geq 2$ and $q \geq 1$.   Let $u_0$, $v_0$ denote the centers of $K_{1,p}$, $K_{1,q}$, respectively, and label the leaves of $K_{1,p}$ and $K_{1,q}$ as $u_1,u_2,\dots,u_p$ and $v_1,v_2,\dots, v_q$, respectively.  We first prove $DD_{\textrm{m}}(K_{1,p} \by K_{1,q}) \leq p+q-1$.  In $K_{1,p} \by K_{1,q}$, we let $$D = \bigcup_{i=1}^{p-1} (u_i,v_0) \cup \bigcup_{i=1}^q (u_p,v_i) \text{~~~and~~~} D' = \bigcup_{i=0}^q (u_0,v_i) \cup \bigcup_{i=1}^{p-2} (u_i,v_q).$$

We first show $(u_a,v_b)$ is dominated by the vertices of $D$.   Observe that $(u_a,v_b)$ is equal or adjacent to $(u_a,v_0) \in D$ if $1 \leq a \leq p-1$ and $0 \leq b \leq q$.   If $a=0$, then $(u_a,v_b)$ is adjacent to $(u_p,v_b) \in D$ if $1 \leq b \leq q$ and adjacent to $(u_1,v_0) \in D$ if $b=0$.  If $a=p$, then $(u_a,v_b) \in D$ if $1 \leq b \leq q$ and adjacent to $(u_p,v_1) \in D$ if $b=0$.  Thus, $D$ dominates $K_{1,p} \by K_{1,q}$.

Note that $(u_a,v_b)$ is equal or adjacent to $(u_0,v_b) \in D'$ if $0 \leq a \leq p$, $0 \leq b \leq q$.  Thus $D'$ dominates $K_{1,p} \by K_{1,q}$.  Observe $D \cap D' = \emptyset$.  Thus, $D,D'$ are disjoint dominating sets.  Finally, we provide a perfect matching between $D$ and $D'$.

For $1 \leq i \leq p-2$, vertex $(u_i,v_0) \in D$ is matched with $(u_i,v_q) \in D'$; vertex $(u_{p-1},v_0) \in D$ is matched with $(u_0,v_0) \in D'$; and for $1 \leq i \leq q$, vertex $(u_p,v_i) \in D$ is matched with $(u_0,v_i) \in D'$.  By inspection, all vertices of $D$ and $D'$ have been matched and note $|D| = |D'| = p+q-1$.\\

%lower bound

We next prove $DD_{\textrm{m}}(K_{1,p} \by K_{1,q}) \geq p+q-1$. For $0 \leq i \leq p$, let $G_i$ denote the subgraph of $K_{1,p} \by K_{1,q}$ induced by vertices $\{(u_i,v_j)~:~ 0 \leq j \leq q\}$; then $G_i \cong K_{1,q}$.  Let $D$, $D'$ be arbitrary disjoint dominating sets with a perfect matching between them such that $|D|=|D'|=DD_{\textrm{m}}(K_{1,p} \by K_{1,q})$.  Since $u_0$ is the center of $K_{1,p}$ and $v_0$ is the center of $K_{1,q}$, we focus on vertex $(u_0,v_0) \in K_{1,p} \by K_{1,q}$.  Observe that $(u_0,v_0)$ is in at most one of $D$, $D'$ (though it may be in neither).  Therefore, without loss of generality, suppose $(u_0,v_0) \notin D$.

%\wfk{I am sorry, but I am getting lost starting with the previous sentence. Why are you picking $(u_0,v_0)$ Is it arbitrary? What if that vertex is neither in $D$ nor $D'$ %(even though it is in D' according to the first part of the proof? That seems like it is allowed, but is it what you want? Maybe I am saying that maybe there are disjoint %dominating sets other than D, D'? Do you need to start this part of the proof with arbitrary disjoint dominating sets with a matching?}\me{I wrote separate proofs and then %pasted them into one, which made things confusing since $D,D'$ are not necessarily the same as in the first part of the proof.  We focus on $(u_0,v_0)$ since it is the %``central'' vertex in the product.  I added some stuff above, so I hope it makes more sense now.  I've also re-written what follows.}\\

If $(u_0,v_i) \in D$ for each $i \in \{1,2,\dots, q\}$, then $G_k$ contains a vertex of $D$ for each $k \in \{1,2,\dots, p\}$ (otherwise $D$ does not dominate $(u_k,v_0)$).  In this case, $|D| \geq p+q$, which exceeds the upper bound obtained in the first part of the proof.

Therefore, excluding $(u_0,v_0)$, $\lambda$ vertices of $G_0$ are not in $D$ for some $\lambda\in \{1,2,\dots,q\}$.
%\wfk{Can we use some symbol other than $\alpha?$}
Without loss of generality, denote these $\lambda$ vertices by $(u_0,v_1),(u_0,v_2),\dots,(u_0,v_\lambda)$.  Then for each $i \in \{1,2,\dots, \lambda\}$, $\exists~ j_i \in \{1,2,\dots,p\}$ such that $(u_{j_i},v_i) \in D$ (otherwise $D$ does not dominate $(u_0,v_i)$).  Then $\cup_{i=1}^\lambda (u_{j_i},v_i) \subseteq V(G_{j_1}) \cup V(G_{j_2}) \dots \cup V(G_{j_p})$.  As it may be that $j_i = j_k$ for $i \neq k$, without loss of generality, let $G_1,G_2,\dots,G_x$ denote the distinct subgraphs of $G_{j_1},G_{j_2},\dots,G_{j_p}$.  As $\lambda \geq 1$, we know that $x \geq 1$.

First, let us assume $x=1$.  Then $(u_1,v_1),(u_1,v_2),\dots, (u_1,v_\lambda) \in D$.  For each $i \in \{2,3,\dots,p\}$, at least one vertex of $G_i$ must be in $D$ (otherwise, as $(u_0,v_0)\notin D$, vertex $(u_i,v_0)$ is not dominated by $D$).  Then $|D| \geq (q-\lambda) + \lambda+(p-1) = p+q-1$ since $q-\lambda$ vertices of $G_0$ are in $D$.

Second, let us assume $x>1$.  For some $k \in \{1,2,\dots,x\}$, suppose $G_k$ contains only one vertex of $D$ (recall $(u_k,v_i) \in D$ for some $i \in \{1,2,\dots, \lambda\}$).  Then for $(u_k,v_j)$, $j \neq i$, to be dominated by $D$, it must be that $(u_0,v_j) \in D$. In this case, $q-1$ vertices of $G_0$ must be in $D$ (else the vertices of $G_k$ are not dominated by $D$) and $|D| \geq p+q$ which exceeds the upper bound obtained in the first part of the proof.  Therefore, for each $k \in \{1,2,\dots,x\}$, $G_k$ must contain at least two vertices of $D$.   Finally, for each $k \in \{x+1,x+2,\dots, p\}$, $G_k$ must contain a vertex of $D$ (otherwise $(u_k,v_0)$ is not dominated by $D$ since $(u_0,v_0) \notin D$).  In this case, $|D| \geq 2x + (q-x) +(p-x) = p+q$ (since there are $q-\lambda$ vertices of $G_0$ in $D$) which exceeds the upper bound obtained in the first part of the proof. Therefore, $|D| \geq p+q-1$.\end{proof}

\smallskip

The next result, Theorem~\ref{thm:treeprod}, shows that for any two non-trivial trees $T$ and $T'$, the product $T \by T'$ has a swap set.  In fact, by the previous result, Theorem~\ref{thm:treeprod} is exact for the product of two stars.  We introduce some notation that will be used in the proof of Theorem~\ref{thm:treeprod}:  for any $X \subseteq V(G)$, let $G[X]$ denote the subgraph of $G$ induced by $X$.  Let $P$ be a partitioning of the vertices of a non-trivial tree $T$ such that each part induces a star of order at least $2$.  Denote by $x_P$, the number of parts in $P$ and by $\ell_P$, the number of leaves in the largest star part in $P$.

\begin{theorem}\label{thm:treeprod} Let $T$ and $T'$ be non-trivial trees. Let $P$ and $P'$ be a partitioning of the vertices of $T$  and $T'$, respectively, such that each part induces a star of order at least $2$.  Then $DD_{\emph{m}}(T \by T') \leq x_{P} x_{P'}(\ell_P+\ell_{P'}-1)$.  \end{theorem}

\begin{proof} Let $P$ ($P'$) be a partitioning of the vertices of non-trivial tree $T$ ($T')$ such that each part induces a star of order at least $2$.  Let $S_1,S_2,\dots,S_{x_P}$ denote the star parts of $P$ and $S_1',S_2',\dots,S_{x_{P'}}'$ denote the star parts of $P'$.  By Theorem~\ref{thm:star}, for each $i \in \{1,2,\dots, x_p\}$, $$DD_{\textrm{m}}(S_i \by K_{1,\ell_{P'}}) = (|S_i|-1) + \ell_{P'}-1 \leq \ell_P + \ell_{P'}-1.$$  As the disjoint union $G[S_1 \by K_{1,\ell_{P'}}]+G[S_2 \by K_{1,\ell_{P'}}]+\dots+G[S_{x_P} \by K_{1,\ell_{P'}}]$ is a spanning subgraph of $T \by K_{1,\ell_{P'}}$, by Lemma~\ref{subgraph}, $DD_{\textrm{m}}(T \by K_{1,\ell_{P'}}) \leq x_P \cdot  (\ell_P + \ell_{P'}-1)$.

Similarly, the disjoint union $G[T \by S_1']+G[T \by S_2']+\dots+G[T \by S_{x_{P'}}']$ is a spanning subgraph of $T \by T'$ and $DD_{\textrm{m}}(T \by S_i') \leq x_P \cdot (\ell_P + \ell_{P'}-1)$.  Thus, by Lemma~\ref{subgraph}, $DD_{\textrm{m}}(T \by T') \leq x_P \cdot x_{P'} \cdot  (\ell_P + \ell_{P'}-1).$\end{proof}

Suppose $G$ and $H$ are non-trivial graphs (i.e. each has at least two vertices). Let $T$ and $T'$ be spanning trees of $G$  and $H$, respectively.  As $T \by T'$ is a spanning subgraph of $G \by H$,  by Lemma~\ref{subgraph}, $DD_{\textrm{m}}(G \by H) \leq DD_{\textrm{m}}(T \by T')$.  Thus, a swap set exists for $G \by H$ regardless of whether swap sets exist for $G$ and $H$.

\begin{corollary}\label{cor:prod} For any non-trivial graphs $G$ and $H$, $DD_{\emph{m}}(G \by H) \leq DD_{\emph{m}}(T \by T')$ where $T$ and $T'$ are spanning trees of $G$ and $H$, respectively.  \end{corollary}

%The {\it strong product} of $G$ and $H,$ written $G \boxtimes H$, has vertex set $V(G)\times V(H)$ and vertices $(a,b)$ and $(c,d)$ are adjacent if $a=c$ and $bd\in E(H)$, $ac\in E(G)$ and $b=d$, or $ac \in E(G)$ and $bd \in E(H)$.  Since $G \by H$ is a spanning subgraph of $G \boxtimes H$, by Lemma~\ref{subgraph} and Corollary~\ref{cor:prod}, if $G \boxtimes H$ also has a swap set regardless of whether swap sets exist for non-trivial graphs $G$ and $H$.

Certainly if graph $G$ has a swap set, then $DD_{\textrm{m}}(G \by H) \leq DD_{\textrm{m}}(G) \cdot |V(H)|$: we can partition the vertices of $G \by H$ into $|V(H)|$ disjoint subgraphs, each isomorphic to $G$.  Thus, if $G$ and $H$ are non-trivial graphs with swap sets, $$DD_{\textrm{m}}(G \by H) \leq \min \Big\{ DD_{\textrm{m}}(G) \cdot |V(H)|, DD_{\textrm{m}}(H) \cdot |V(G)|\Big\}.$$  Although, $DD_{\textrm{m}}(G \by H) \geq \gamma(G \by H)$, we ask if this bound could be improved if the swap sets and domination numbers of the input graphs $G$ and $H$ are exploited.

\begin{question} Let $G$ and $H$ be non-trivial graphs, each containing a swap set. Is it true that $$DD_{\emph{m}}(G \by H) \geq \min\Big\{DD_{\emph{m}}(G) \cdot \gamma(H), \gamma(G) \cdot DD_{\emph{m}}(H))\Big\}?$$ %provided the right hand side of the inequality is finite?
\end{question}

Recall that Vizing's Conjecture states that $\gamma(G \by H) \geq \gamma(G)\cdot \gamma(H)$; it has been verified for some classes of graphs (including trees), but remains unsettled for general graphs $G$ and $H$.

\begin{question} Is it true for all $G, H$, that $DD_{\emph{m}}(G \by H) \geq  \gamma(H) \cdot \gamma(G)$?
\end{question}

%%%%%%%%%%%%%%%%%%%%%%%%%%%%%%%%%%%
\subsection{Grid Graphs}\label{subsec:mn}

In~\cite{EvictionPaper}, $DD_{\textrm{m}}(P_2 \by P_n)$ and $DD_{\textrm{m}}(P_4 \by P_n)$ were determined for all $n \geq 1$ and $DD_{\textrm{m}}(P_3 \by P_n)$ was determined for $n \not\equiv 1$ (mod $4$).  An open question of~\cite{EvictionPaper} was: is it true that $DD_{\textrm{m}}(P_3 \by P_n) = \gamma(P_3 \by P_n)+2$ when $n \equiv 1$ (mod $4$) for $n > 9$?  We next answer the question negatively.

\begin{lemma}
$DD_{\emph{m}}(P_3 \by P_n)=\gamma(P_3 \by P_n)+1$ for $n \equiv 1 \pmod{4}$ and $n > 9$.
\end{lemma}

\begin{proof}
Let $n=4k + 1$ for some integer $k \geq 3$.  For a lower bound, we use the domination number from~\cite{JK} and the result of Corollary 11 from~\cite{EvictionPaper}, that $e_{\textrm{m}}^\infty(P_3 \by P_{4k+1})  = \gamma(P_3 \by P_{4k+1})+1$.  Thus, $DD_{\textrm{m}}(P_3 \by P_{4k+1}) \geq 3k+2$.  To achieve the upper bound, we let $D$ and $D'$ denote the squared vertices in Figure~\ref{fig:4k+1} (a) and (b) respectively.  Observe that $D$ and $D'$ are disjoint dominating sets with a perfect matching between them and that $|D| = |D'| = 3k+2$ as desired. \end{proof}

\begin{figure}[htbp]\label{fig:3by4k+1}
\[ \includegraphics[width=0.8\textwidth]{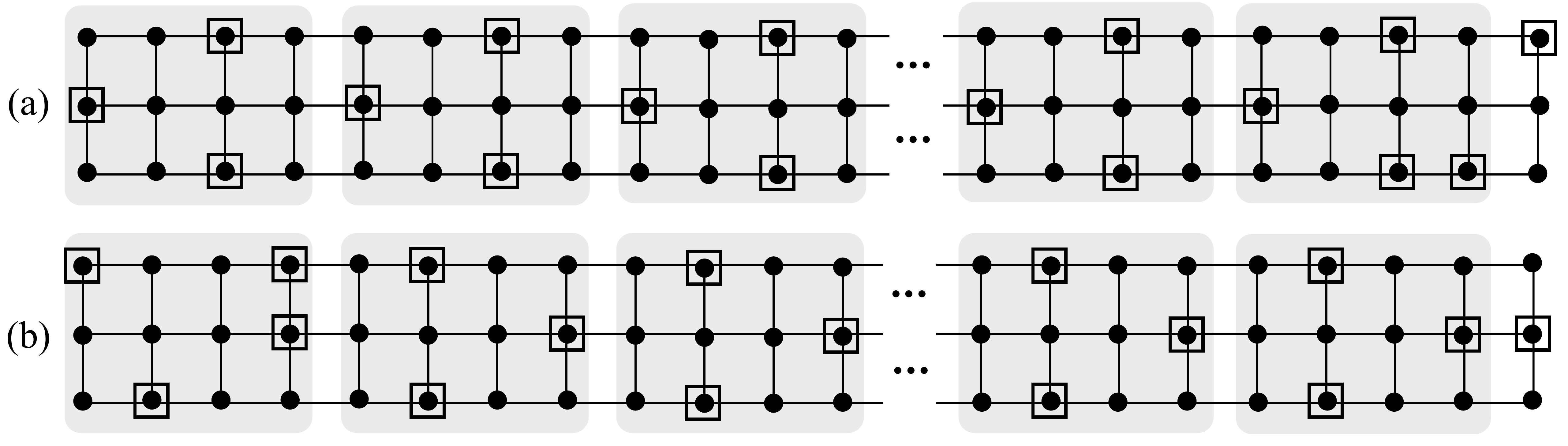} \]

\caption{Minimum disjoint dominating sets on $P_3 \by P_{4k+1}$.}

\label{fig:4k+1}

\end{figure}

Applying Theorem~\ref{thm:treeprod} to paths $P_n$ and $P_m$, we find $DD_{\textrm{m}}(P_n \by P_m) \leq \frac{nm}{3} + o(n+m)$.  However, this bound is not tight.  From~\cite{dom, JK}, we know that $\gamma(P_m \by P_n) = \frac{mn}{5} + o(m+n)$ for $m \geq n \geq 8$; thus it remains to be shown that $DD_{\textrm{m}}(P_m \by P_n) \leq \frac{mn}{5}+o(m+n)$ for $m \geq n \geq 8$.  In particular, we prove the following:

\begin{theorem}\label{uppergrid} For $m \geq n \geq 8$, $DD_{\emph{m}}(P_m \by P_n) \leq \lfloor \frac{(n+2)(m+3)}{5}\rfloor$.\end{theorem}

The remainder of this subsection is devoted to providing disjoint dominating sets $D$ and $D'$ for $P_m \by P_n$ that have a perfect matching and for which $|D|=|D'| = \lfloor \frac{(n+2)(m+3)}{5}\rfloor$.\\

In \cite{EvictionPaper} it was proven that $e^\infty_m(P_n \by P_m) \leq \lfloor\frac{(n+2)(m+3)}{5}\rfloor-4$. The proof consisted of an algorithmic eviction strategy to defend an $n \times m$ grid graph against any sequence of attacks.  We modify the strategy to provide disjoint dominating sets of size $\lfloor\frac{(n+2)(m+3)}{5}\rfloor$ with a perfect matching.

A {\it perfect dominating set} is a set $S\subseteq V$ such that for all $v \in V$, $|N[v] \cap S|=1$. We use the description given in~\cite{Chang} of a perfect dominating set on an infinite grid graph where the vertices are labeled according to their Cartesian coordinates: for any $t \in \{0,1,2,3,4\}$, the vertices in a perfect dominating set are given by the set \begin{equation}\label{perf} S_t = \Big\{ (x,y)~|~y=\frac{1}{2}x+\frac{5}{2}s+t \text{ and } x,y \in \mathbb{Z}, s \in \mathbb{Z}\backslash\{0\}\Big\}.\end{equation}  For $t=3$, the vertices of a perfect dominating set of the infinite grid graph are indicated in Figure~\ref{FigY1} (changing the value of $t$ simply translates the dominating set).

\begin{figure}[htbp]
\[ \includegraphics[width=0.35\textwidth]{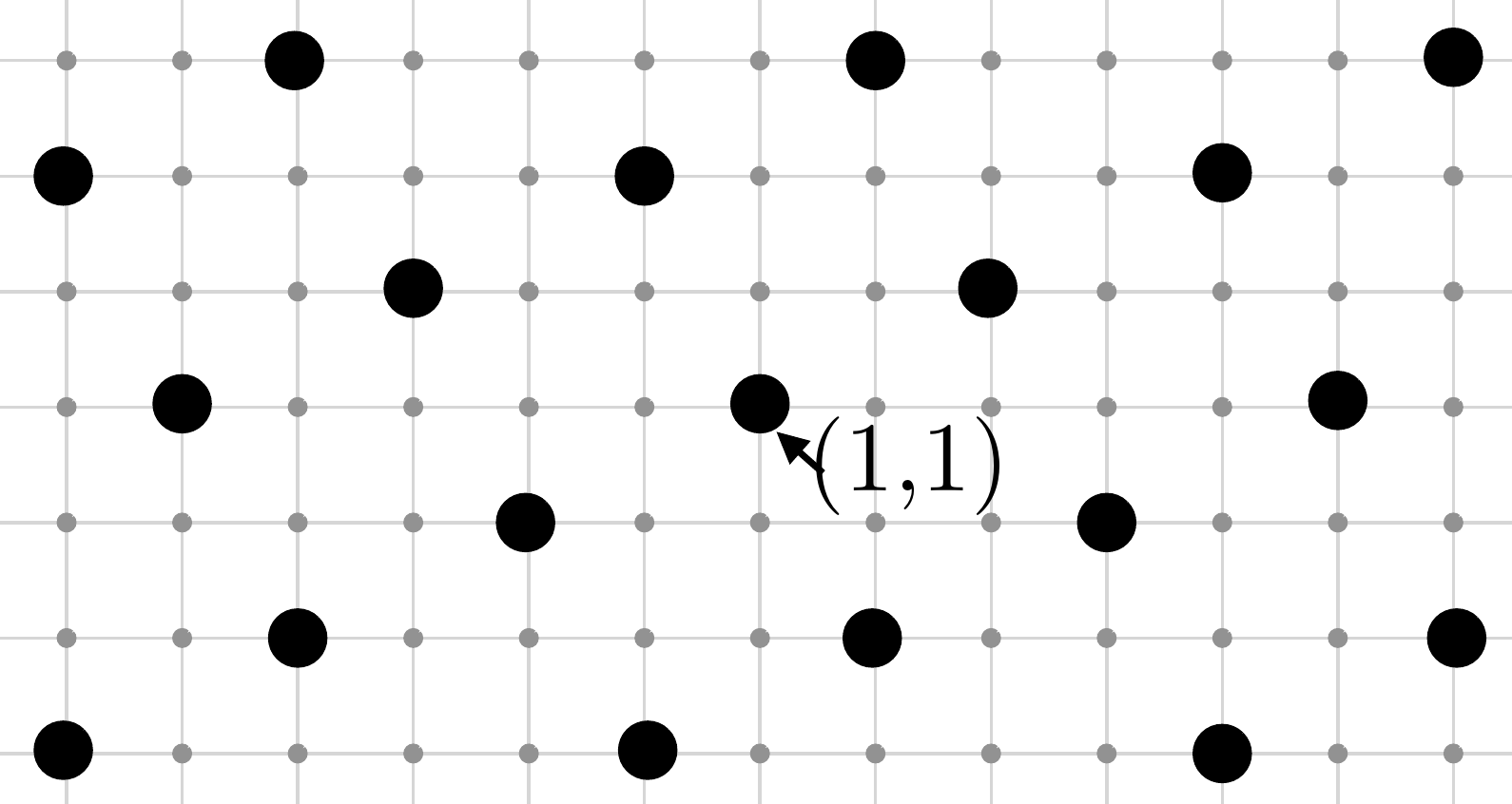} \]

\caption{A perfect dominating set on the infinite grid. }

\label{FigY1}

\end{figure}

Consider the sub-grid $P_m \by P_n$ of the infinite grid graph, induced by vertices $\{ (i,j)~|~ 1 \leq i \leq m, 1 \leq j \leq n\}.$  We will place black and white tokens on vertices of $P_m \by P_n$ and these tokens will together, initially identify the vertices of dominating set $D$.  (We use two colors of tokens as we later have two rulesets for moving tokens to form $D'$.) For $i \in [m-1]$, $j \in [n]$, if vertex $(i,j)$ is in the perfect dominating set of the infinite grid graph defined by $S_3$ in~(\ref{perf}), then a black token is placed on $(i,j)$ in $P_m \by P_n$.  Such a vertex will be referred as a {\it black vertex}.  Observe that no two black vertices are adjacent.  For $j \in [n]$, if vertex $(m,j)$ is in the perfect dominating set of the infinite grid graph defined by $S_3$ in~(\ref{perf}), then a white token is placed on $(m,j)$ in $P_m \by P_n$. Such a vertex will be referred as a {\it white vertex}.

For the remainder of the proof, we refer to vertex $(i,j)$ as being in {\it column $i$} and {\it row $j$} (following the Cartesian coordinate convention). Some vertices in row $1$ and row $n$ of $P_m \by P_n$ are not yet dominated by vertices with tokens.  For $i \in [m]$, if vertex $(i,n+1)$ or $(i,0)$ is in the perfect dominating set $S_3$ defined in~(\ref{perf}), then we place a black token on vertex $(i,n)$ or $(i,1)$, respectively. Some  vertices in column $1$ and column $m$ of $P_m \by P_n$ are not yet dominated by vertices with tokens.  For $ j \in [n]$, if vertex $(0,j)$ or $(-1,j)$ is in the perfect dominating set $S_3$ defined by~(\ref{perf}), then we place a white token on vertex $(2,j)$, provided that vertex does not already have a token (we consider vertices of the form $(-1,j)$ in addition to $(0,j)$ because later, most guards will move to the right to form $D'$ and the additional guards will ensure $D'$ forms a dominating set).  For $j \in [n]$, if vertex $(m+1,j)$ is in the perfect dominating set $S_3$ defined by~(\ref{perf}), then we place a white token on vertex $(m,j)$, provided that vertex does not already have a token.  If $(0,n+1)$, $(m+1,n+1)$, or $(m+1,0)$ is in the perfect dominating set $S_3$, then we place a white token at $(2,n)$, $(m,n)$, or $(m,0)$, respectively.   An example of the placement of black and white tokens on $P_{16} \by P_{12}$ is illustrated in Figure~\ref{fig2} (a); note that the arrows in this figure indicate how vertices in the perfect dominating set of $S_3$ in column $m$ and outside the bounds of $P_m \by P_n$ are mapped to vertices in $P_m \by P_n$.

\begin{figure}[htbp]
\[ \includegraphics[width=\textwidth]{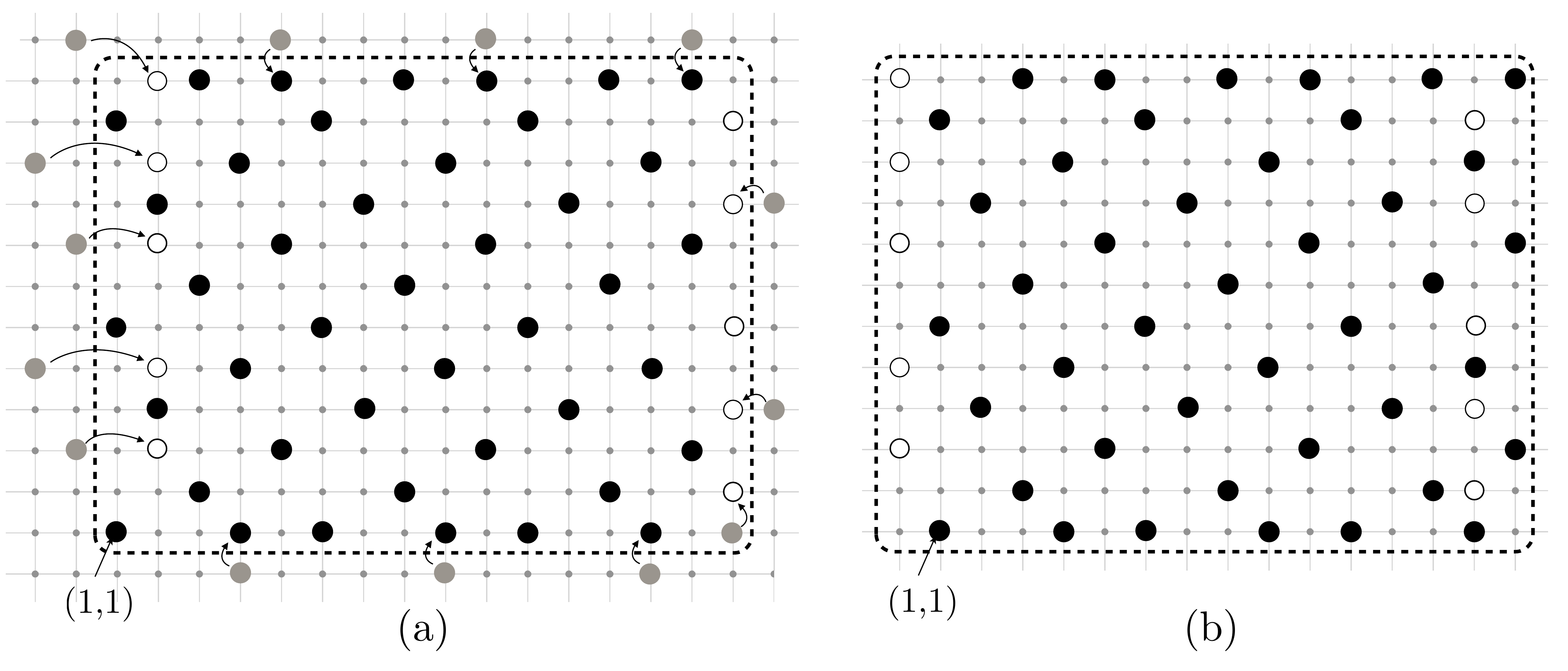} \]

\caption{For $P_{16} \by P_{12}$, disjoint dominating set $D$ is given in (a) and $D'$ is given in (b).}

\label{fig2}

\end{figure}

Since we are considering the infinite grid graph defined by $S_3$, the vertex $(1,1)$ in $P_m \by P_n$ will always have a token.  However, for the three remaining ``corners'', there are five possible arrangement of tokens, depending on which, if any, of any of the four corner vertices (that form a $4$-cycle) are in the perfect dominating sets on the corresponding infinite grid graph defined by $S_3$.  For example, consider the top left corner of $P_{16} \by P_{12}$ shown in Figure~\ref{fig2} (a).  It is necessary to place a white token on $(2,n)$: therefore $(2,n) \in D$; otherwise, vertex $(1,n)$ will not be dominated by $D'$.  Similarly, consider the bottom right corner of $P_{16} \by P_{12}$ shown in Figure~\ref{fig2} (a).  If there was a white token on $(m,1)$, then we would not have a perfect matching with $D'$ as tokens from $(m-2,1)$ and $(m,1)$ would both move to $(m-1,1)$. Therefore, instead of placing a white token on $(m,1)$, we place a white token on $(m,2)$ (then $(m,2) \in D$). Then vertices $(m,1), (m,2)$ are dominated by $D$ and no conflict arises.  In the majority of situations, like the top right corner of $P_{16} \by P_{12}$ shown in Figure~\ref{fig2} (a), there is no need to deviate from the original placement of tokens.  Since there are three corners to consider and five possible situations for each, we omit the remaining situations as the deviations are similar to the one described above.

By the placement of tokens, the vertex immediately to the right of each black vertex does not contain a token; and the vertex immediately to the left of each white vertex does not contain a token.  Thus, to move from $D$ to disjoint set $D'$, we simply move each black token immediately to the right and each white token immediately to the left.  The set of vertices that contain black and white tokens after this transition form a new dominating set $D'$; and observe that there is a perfect matching between dominating sets $D$ and $D'$.

Since $|D|=|D'| = \lfloor \frac{(n+2)(m+3)}{5}\rfloor$, we find $DD_{\textrm{m}}(P_n \by P_m) \leq  \lfloor \frac{(n+2)(m+3)}{5}\rfloor$ as desired.

\section{Independence Number}\label{sec:ind2}

%%%%%%%%%%%%%%%%%%%%%%%%%%%%%%%%%%%

It was shown in~\cite{EvictionOnTrees}, that $\gamma(G) \leq e_{\textrm{m}}^\infty(G) \leq \alpha(G)$ for any graph $G$.
Although $e_{\textrm{m}}^\infty(G) \leq DD_{\textrm{m}}(G)$ and  there exist many graphs (see Figures~\ref{fig} and~\ref{Fig2} for example) for which $DD_{\textrm{m}}(G) \leq \alpha(G)$, we do not know in general whether $DD_{\textrm{m}}(G)$ and $\alpha(G)$ are comparable.  In this section, we provide results for graphs with small independent number.

\begin{theorem}\label{ind2} Let $G$ be a connected graph with $n > 3$ vertices such that $\alpha(G)=2$. Then $G$ contains a swap set and furthermore,  $DD_{\emph{m}}(G) \leq 2$.
\end{theorem}
\begin{proof} Let $I=\{u, v\}$ be a maximum independent set in $G$. Since $n > 3$, $|N(I))| \geq 2$. First suppose that $N(I)$ contains an independent set of size two, say $\{y, z\}$. Then there is a matching between $\{u, v\}$ and $\{y, z\}$, else $G$ contains an independent set of size three. Since $\{y, z\}$ is a dominating set, $G$ contains a swap set. On the other hand, if $N(I)$ is complete, since $|N(I)| \geq 2$ and since $G$ is connected, there exist two vertices $\{w, x\} \subseteq N(I)$ such that $\{w, x\}$ dominate $G$ and such that there exists a matching between $\{u, v\}$ and $\{w, x\}$.
\end{proof}

We leave as an open problem the characterization of those graphs $G$ for which $DD_{\textrm{m}}(G) = 2$.

In what follows, we say that a matching $M= \{u_1v_1, u_2v_2, \ldots, u_kv_k\}$ is a {\it desired matching} if both $\{u_1, u_2, \ldots u_k\}$ and  $\{v_1, v_2, \ldots v_k\}$
are dominating sets.

\begin{theorem} Let $G$ be a connected graph containing a swap set such that $\alpha(G)=3$. Then $DD_{\emph{m}}(G) \leq \alpha(G)$.
\end{theorem}

\begin{proof} Let $I$ be a maximum independent set of $G$, so $|I|=3$. First suppose there exists a set of vertices $J$ such that there is a perfect matching between $I$ and $J$. Assume that $I$ and $J$ are chosen so that $J$ dominates as many vertices as possible, over all choices of $I, J$ where $I$ is a maximum independent set and there is a perfect matching $M$ between $I$ and $J$. Suppose to the contrary that $I$ is not a swap set. Since $I$ is a dominating set, it must be that $J$ is not a dominating set. Let $x \in V \backslash (I \cup J)$ be a vertex not dominated by $J$. Let $x$ be dominated by $v \in I$. Since $|J|=3$ and $x$ is not dominated by $J$, it must be that there is an edge $yz$, where $\{y, z\} \in J$ (otherwise $J$ is a maximum independent set of $G$ and thus a dominating set).

Let us first suppose $v$ is matched with $z$ in $M$. Then $J' = J \cup \{x\}\backslash \{z\}$ contradicts the maximality of $I, J$, unless $z$ is the only vertex in $J$ adjacent to $a \in V\backslash (I \cup J)$ and $x$ is not adjacent to $a$ (as $\alpha(G)=3$, there is at most one such vertex $a$).  Suppose $z$ is the only vertex in $J$ adjacent to $a \in V \backslash (I \cup J)$ and $x$ is not adjacent to $a$.  Let $w$ be a vertex of $J'$ (along with $x$ and $y$); since $x$ is not adjacent to $y$ or $w$, it must be that $w$ is adjacent to $y$ (else $J'$ is a maximum independent set and thus a dominating set).  Let $c, d$ be the other two vertices in $I$, with $cw, dy \in M$. If $c$ ($d$) is adjacent to $a$, then $\{ca, dy, vx\}$ ($\{da, cw, vx\}$) is the desired matching (note that both $\{y, x, a\}$ and $\{w, x, a\}$ are independent sets). Otherwise, neither $c$ nor $d$ are adjacent to $a$. But this implies $v$ is adjacent to $a$ (since $I$ is a dominating set). If $x$ is independent of $\{c, d\}$, then $\{x, a, c, d\}$ is an independent set of size four. Therefore, without loss of generality, that $d$ is adjacent to $x$. Then  $\{cw, dx, va\}$ is the desired matching (note that $\{w, x, a\}$ is an independent set).

%Then $\{cw, dy, vx, za\}$ is a desired matching, as both $\{x, y, a\}$ or $\{x, w, a\}$ are independent sets (and thus dominating sets).
%\me{I agree with that this sentence is true - but I 5don't see how it ends the case (i.e. how it shows $DD \leq \alpha$).  I must be missing something simple?}

%MEM"s old attempt
%We now swap $a$ and $w$:  let $J'' = J\cup \{a\} \backslash \{w\}$.  Then $J''$ contradicts the maximality of $I, J$, unless there is a vertex $b$ in $V \backslash (I \cup J'')$ that is %adjacent to $w$ but not adjacent to $x$, $y$, or $a$.  This cannot happen because $x,y,a,b$ then form an independent set of cardinality four.
%We need to ensure that there is a perfect matching between $I$ and $J''$.  Let $c$ and $d$ be the other two vertices in $I$ where $c$ is matched with $y$ via $M$ and $d$ is matched with %$w$ via $M$. If $d$ is adjacent to $a$ then $\{vz, da, cy\}$ is the desired matching between $I$ and $J''$. Suppose $d$ is not adjacent to $a$. Then $\{d, y, a\}$ is an independent set %*and thus a dominating set). STUCK
%\me{Are we sure we have a perfect matching between $I$ and $J''$? I don't know how to explain this.}

% Next is hopefully fixed:

Therefore, suppose $v$ is not matched with $z$ (or, by the same logic, $y$), but some other vertex $w$ that is independent of the remaining two vertices in $J$. If $x$ is adjacent to either of the other two vertices in $I$, then the argument above applies. So assume $x$ is independent of both of the other two vertices in $I$. Observe $w$ is adjacent to at least one of $c,d$; otherwise $\{c,d,w,x\}$ is an independent set.  If $w$ is adjacent to $c$ then $\{y,w,x\}$ is an independent set and $\{dy,cw,vx\}$ is the desired matching.  If $w$ is adjacent to $d$ then $\{w,z,x\}$ is an independent set and $\{dw, cz, vx\}$ is the desired matching.

Now let us suppose there is no such set $J$ such that there is a perfect matching between $I$ and $J$. Let $M$ be a matching that matches the maximum number of vertices in $I$. Let $J$ be the set of endvertices of $M$ not contained in $I$. Then there must be some vertex $v \in I$ that is not matched by $M$. There are two possibilities. If $V = I \cup J$ then since $|J| \leq 2$, either $G$ has no swap set (a contradiction) or $DD_{\textrm{m}}(G) \leq 2$.

%Otherwise there is a vertex $x \notin I \cup J$. Let $I'$ be the subset of $I$ containing vertices in $I$ not matched by $M$. Then $x$ is not adjacent to any vertex in $I'$, %else the maximality of $M$ is contradicted.  Since $I$ is a dominating set, $x$ must be adjacent to some vertex $u \in I$. If $u$ is matched to $z \in J$ where $z$ is adjacent %to some vertex $v \in I'$, then a larger matching than $M$ can be found (remove $uz$ from $M$ and replace with $vz$ and $ux$). Since this contradicts the maximality of $M$, it %must be that $x$ is adjacent to $u \in I$, $z$ is matched to $u$ and one of the following situations occurs. Either (A) $v$ is adjacent to some vertex $y \in J$, where $y$ is %matched to $q \in I$ (and $q$ is not adjacent to $x$). Then either $y$ is adjacent to two pendant vertices ($q$ and $v$) which implies $G$ has no swap set; or $q$ is adjacent %to $z$ (in which case we form the matching $qz$, $yv$, $ux$); or $q$ is adjacent to some other vertex $x' \notin I \cup J$, which also implies the existence of a matching %larger than $M$. Otherwise, we have (B) if $|I'|=2$ then $z=y$ and $u=q$ and so $q$ is adjacent to $x$ (could also have $z=y$, $u=q$ if $|I'|=1$. Then either $y$ is adjacent to %two pendant vertices ($q$ and $v$) which implies $G$ has no swap set; or $q$ is adjacent to $z$ (in which case we form the matching $qz$, $yv$, $ux$); or $q$ is adjacent to %some other vertex $x' \notin I \cup J$, which also implies the existence of a matching larger than $M$.

Otherwise there is a vertex $x \notin I \cup J$. Let $I'$ be the subset of $I$ containing vertices in $I$ not matched by $M$.  Then $x$ is not adjacent to any vertex in $I'$, else the maximality of $M$ is contradicted.   Since $I$ is a dominating set, $x$ must be adjacent to some vertex $u \in I$.  First, suppose $|I'|=1$.  If $u$ is matched to $z \in J$ where $z$ is adjacent to some vertex $v \in I'$, then a larger matching than $M$ can be found (remove $uz$ from $M$ and replace with $vz$ and $ux$). Since this contradicts the maximality of $M$, it must be that $x$ is adjacent to $u \in I$, $z$ is matched in $u$ and $v \in I'$ is adjacent to some vertex $y \in J$, where $y$ is matched to $q \in I$ (and $q$ is not adjacent to $x$). Then either $y$ is adjacent to two pendant vertices ($q$ and $v$) which implies $G$ has no swap set; or $q$ is adjacent to $z$ (in which case we form the matching $qz$, $yv$, $ux$); or $q$ is adjacent to some other vertex $x' \notin I \cup J$, which also implies the existence of a matching larger than $M$.  Second, suppose $|I'|=2$.  If $v \in I'$ is adjacent to $z \in J$, then we contradict the maximality of $M$ (remove $uz$ from $M$ and replace it with $ux$ and $vz$).  Thus, the vertices of $I'$ are isolated vertices which contradicts the requirement that $G$ is connected.

\end{proof}

\begin{conjecture}
Let $G$ be a connected graph containing a swap set. Then $DD_{\emph{m}}(G) \leq \alpha(G)$.
\end{conjecture}

We make the following conjecture, which is similar in nature to the well-known theorem that the order of a graph equals its matching number plus its independence number. However, this conjecture is a bit more involved than that, as it requires the existence of a matching that is a desired matching.

\begin{conjecture} \label{swap-conj} For any constant $x$ there is a positive constant $y$ such that if $\alpha(G)=x$ and $|V(G)| \geq y$, then $G$ has a swap set.
\end{conjecture}

Repeating an earlier example, consider $K_3$ with each edge duplicated and then subdivided.  Call this new graph $G$.  Then $\alpha(G) = 6$, $|V(G)| = 9$, but $G$ has no swap set. Thus if Conjecture \ref{swap-conj} is true, then if $\alpha(G)=6$, we need $y$ to be at least 10 (and perhaps larger) in order to ensure that any $G$ with $\alpha(G)=6$ and $|V(G)|= y$ vertices has a swap set.

\begin{theorem}\label{ind3} Let $G$ be a connected graph with $n \geq 6$ vertices such that $\alpha(G)=3$. Then $G$ contains a swap set.
\end{theorem}

\begin{proof} Let $I=\{u, v, w\}$ be a maximum independent set in $G$. Since $n \geq 6$, $|N(I))| \geq 3$. Let $M$ be a matching whose endvertices include $I$; such a matching exists since $|V(G)| \geq 6$. Let $J=\{a, b, c\}$ be the set of vertices matched with $I$ by $M$, where $M=\{ub, vc, wa\}$.
If $J$ is a dominating set, we are done. So suppose $J$ is not a dominating set and let $Q$ be the set of vertices not dominated by $J$ (clearly, $Q \cap I = \emptyset$). Since we can remove (pairs of) vertices from any matching in the subgraph induced by $Q$  (one vertex from each edge of such a matching can be added to a swap set that includes the vertices in $I$), we may henceforth assume that $Q$ is an independent set. Since $J$ and $Q$ are not joined by any edges, it follows that $1 \leq |Q| \leq 2$.  Each vertex in $Q$ is adjacent to at least one vertex in $I$. If $Q$ is non-empty, there must be at least one edge in the subgraph induced by the vertices of $J$, else $\alpha(G) > 3$.

First suppose $|Q|=2$. Let $y, z$ be the vertices in $Q$. Then if both vertices in $Q$ are only adjacent to one vertex in $I$, say $u$, then $\alpha(G)=4$. So suppose $uy, vz$ are edges.
Let us first suppose that $ab$ is an edge. We construct the following matching $M'$: $ab, uy, vz$. If $cw$ is not an edge and if $w$ is independent of $y, z$, then $c, w, y, z$ is an independent set of size four, a contradiction. If $cw$ is an edge, then we can add that to $M'$ and we have the desired matching. Otherwise, $w$ must have an edge to either $y$ or $z$. If $w$ is adjacent to $z$, then we can extract the desired matching from the $P_8$: $c, v, z, w, a, b, u, y$ (i.e. $\{c,w,u\}$ and $\{v,a,y\}$ are disjoint dominating sets with a perfect matching between them). So suppose that $w$ is not adjacent to $z$, rather $w$ is adjacent to $y$.
If $u$ is adjacent to neither $z$ nor $c$, then $\{w, u, z, c\}$ is an independent set. But if $u$ is adjacent to either $z$ nor $c$, then a $P_8$ again exists (either $z, v, c, u, b, a, w, y$ or $c, v, z, u, b, a, w, y$) from which the desired matching can be extracted.

Now suppose $ab$ is not an edge (and similarly, $ac$ is not an edge). So $bc$ must be an edge. Now the matching $\{wa, vz, bc, uy\}$ yields the desired swap set (take either set of endvertices of this matching).

%%Can avoid the next case if we settle for $n \geq 8$

Finally suppose $|Q|=1$.
Let $x$ be the vertex in $Q$ and suppose $x$ is adjacent to $u$.  Recall there must be at least one edge in the subgraph induced by $J$, otherwise $\alpha(G)>3$.  If $b$ adjacent to $a$ or $c$ then the subgraph induced by $I \cup J \cup Q$ contains one of the following configurations: (i) a $P_7$ subgraph; or (2) disjoint $P_5$ and $P_2$ subgraphs.  (The subgraphs specified in (i) and (ii) are not necessarily induced.)  Otherwise, $a$ must be adjacent to $c$, else $\alpha(G)>3$.  If one of  $xv, xw, bv, bw$ exist, then $I \cup J \cup Q$ contains (i); if none of the four edges exist, then $\alpha(G)>3$.  Since both $P_7$ and the disjoint union of $P_5$ and $P_2$ contain swap sets, we are done.\end{proof}

\section{Acknowledgements}
M.E.~Messinger acknowledges research support from NSERC (DDG-2016-00017) and Mount Allison University.

%%%%%%%%%%%%%%%%%%%%%%%%%%%%%%%%%%%

\end{document}